\documentclass[12pt,reqno]{amsart}
\usepackage{latexsym,amsmath,amsfonts,amssymb,amsthm}
\theoremstyle{plain}
\newtheorem{Thm}{Theorem}

\newtheorem{Conj}{Conjecture}
\theoremstyle{definition}
\newtheorem*{Ack}{Acknowledgment}
\theoremstyle{remark}
\newtheorem*{Rem}{Remark}

\def\N{\mathbb N}

\def\P{\mathcal P}
\def\pmod #1{\ ({\rm mod}\ #1)}

\begin{document}
\title{The Romanov theorem revised}
\author{Hongze Li}
\email{lihz@sjtu.edu.cn}
\author{Hao Pan}
\email{haopan79@yahoo.com.cn}
\address{
Department of Mathematics, Shanghai Jiaotong University, Shanghai
200240, People's Republic of China} \maketitle
\begin{abstract}
Let $\P$ be the set of all primes and $\P_2=\P\cup\{p_1p_2:\,
p_1,p_2\in\P\}$. We prove that the sumset
$$
2^{\P}+\P_2=\{2^{p}+q:\, p\in\P, q\in\P_2\}
$$
has a positive lower density.
\end{abstract}

For a subset $A$ of positive integers, define $A(x)=|\{1\leq a\leq
x:\, a\in A\}|$. Let $\P$ denote the set of all primes and $
2^{\N}=\{2^n:\, n\in\N\}$, where $\N=\{0,1,2,\ldots\}$. A
classical result of Romanov \cite{Romanov34} asserts that the
sumset
$$
2^{\N}+\P=\{2^n+p:\, n\in\N, p\in\P\}
$$
has a positive lower density, i.e., there exists a positive
constant $C_R$ such that $(2^{\N}+\P)(x)>C_Rx$ for sufficiently
large $x$. Recently, the lower bound of $C_R$ has been calculated
in \cite{ChenSun04, HabsiegerRoblot06, Pintz06}. Now let
$$
\P_2=\{q:\, q\text{ is a prime or the product of two primes}\}.
$$
Motivated by Romanov's theorem, in this short note we shall show
that:
\begin{Thm}
The sumset
$$
2^{\P}+\P_2=\{2^{p}+q:\, p\in\P, q\in\P_2\}
$$
has a positive lower density.
\end{Thm}
\begin{proof}
In our proof, the implied constants by $\ll$, $\gg$ and $O(\cdot)$
will be always absolute.

For $q\in\P_2\setminus\P$, let $\psi(q)$ be the least prime factor
of $q$. And we set $\psi(p)=1$ if $p\in\P$. Let
$$
\P_2^*=\{q\in\P_2:\, \psi(q)<q^{\frac{1}{3}}\}.
$$
It suffices to show that $2^{\P}+\P_2^*$ has a positive lower
density.

In view of the Chebyshev theorem, we have
\begin{align*}
\P_2^*(x)=&|\{(p_1,p_2):\,p_1\in\P\cup\{1\}, p_2\in\P, p_1^2<
p_2\leq x/p_1\}|\\
\geq&\sum_{\substack{p_1\in\P\\ p_1 \leq
x^{\frac{1}{3}}}}\bigg(\frac{x/p_1}{5\log(x/p_1)}-\frac{5p_1^2}{\log(p_1^2)}\bigg)
\\
\geq&\frac{x}{5\log x}\sum_{\substack{p_1\in\P\\ p_1 \leq
x^{\frac{1}{3}}}}\frac{1}{p_1}-\frac{x^{\frac{1}{3}}}{\log(x^{\frac{1}{3}})}\cdot\frac{5x^{\frac{2}{3}}}{\log(x^{\frac{2}{3}})}.
\end{align*}
It follows that $\P_2^*(x)\gg x\log\log x/\log x$ since
$$
\sum_{p\in\P\cap[1,x]}\frac{1}{p}=(1+o(1))\log\log x.
$$
Similarly it is not difficult to deduce that $\P_2(x)\ll x\log\log
x/\log x$. Let
$$
r(n)=|\{(p,q):\, n=2^{p}+q,\ p\in\P,\ q\in\P_2^*\}|.
$$
Clearly we have
\begin{align*}
\sum_{n\leq x}r(n)=&|\{(p, q):\, p\in\P, q\in\P_2^*,
2^{p}+q\leq x\}|\\
\geq&2^{\P}(x/2)\P_2^*(x/2)\\
\gg&\frac{\log x}{\log\log x}\cdot\frac{x\log\log x}{\log x}=x.
\end{align*}
And by Cauchy-Schwarz's inequality,
$$
\bigg(\sum_{n\leq x}r(n)\bigg)^2\leq(2^{\P}+\P_2^*)(x)\sum_{n\leq
x}r(n)^2.
$$
Therefore we only need to prove that
\begin{equation}
\label{r2} \sum_{n\leq x}r(n)^2 =|\{(p_1,p_2,q_1,q_2):\,
p_1,p_2\in\P,\ q_1,q_2\in\P_2^*,\ 2^{p_1}+q_1=2^{p_2}+q_2\leq x\}|
\end{equation}
is $O(x)$.

Below we shall show that the following
\begin{equation}
\label{selbergp2} |\{q\leq x-N:\, q,
q+N\in\P_2^*\}|\ll\frac{x(\log \log x)^2}{(\log x)^2}\prod_{p\mid
N}\bigg(1+\frac{1}{p}\bigg)
\end{equation}
is hold for arbitrary positive even integer $N$. Define
$$
\mathfrak{S}(n)=\prod_{p\mid n}\bigg(1+\frac{1}{p}\bigg).
$$
As an application of Selberg's sieve method (cf. \cite[Sections
 7.2 and 7.3]{Nathanson96}), we know that
\begin{equation}
\label{selberg} |\{1\leq n\leq x:\, k_1n+l_1,
k_2n+l_2\in\P\}|\ll\frac{x}{(\log x)^2}\mathfrak{S}
(k_2l_1-k_1l_2).
\end{equation}
for non-negative integers $k_1,k_2,l_1,l_2$ with $(k_i,l_i)=1$ and
$2\mid k_2l_1-k_1l_2$. Observe that $n, n+N\in\P_2^*$ if and only
if there exist $p_1, p_2\in\P$ such that $n/p_1, (n+N)/p_2\in\P$.
Assume that $n/p_1=p_2m+l$ where $1\leq l\leq p_2$. Then
$$
(n+N)/p_2=(p_1p_2m+p_1l+N)/p_2=p_1m+(p_1l+N)/p_2,
$$
whence $p_1l\equiv -N\pmod{p_2}$. Note that $l$ is uniquely
determined by $p_1$ and $p_2$ unless $p_1=p_2$. Thus
\begin{align*}
&|\{n\leq x:\, n, n+N\in\P_2^*,\ p_1\mid n,\ p_2\mid (n+N)\}|\\
\leq&\begin{cases} |\{m\leq x/p_1:\, m, m+N/p_1\in\P\}|&\text{if }
p_1=p_2,\\
|\{m\leq x/p_1p_2:\, p_2m+l,
p_1m+(p_1l+N)/p_2\in\P\}|&\text{otherwise}.
\end{cases}\\
\ll&\begin{cases} \frac{x/p_1}{(\log(x/p_1))^2}\mathfrak{S}
(N/p_1)&\text{if }
p_1=p_2\mid N,\\
\frac{x/p_1p_2}{(\log(x/p_1p_2))^2}\mathfrak{S}
(N)&\text{otherwise}.
\end{cases}\\
\end{align*}
Therefore
\begin{align*}
&|\{q\leq x-N:\, q, q+N\in\P_2^*\}|\\
\ll&\sum_{\substack{p_1, p_2\in\P\\
p_1,p_2\leq
x^{\frac{1}{3}}}}\frac{x/p_1p_2}{(\log(x/p_1p_2))^2}\mathfrak{S}
(N)+\sum_{\substack{p\in\P\\p\mid N, p\leq
x^{\frac{1}{3}}}}\frac{x/p}{(\log(x/p))^2}\mathfrak{S}(N/p).
\end{align*}
Now
$$
\sum_{\substack{p_1, p_2\in\P\\
p_1,p_2\leq
x^{\frac{1}{3}}}}\frac{x/p_1p_2}{(\log(x/p_1p_2))^2}\leq\frac{9x}{(\log x)^2}\sum_{\substack{p_1, p_2\in\P\\
p_1,p_2\leq x^{\frac{1}{3}}}}\frac{1}{p_1p_2}\ll\frac{x(\log\log
x)^2}{(\log x)^2}.
$$
And
$$
\sum_{\substack{p\in\P\\p\mid N, p\leq
x^{\frac{1}{3}}}}\frac{x/p}{(\log(x/p))^2}\leq\sum_{\substack{p\in\P\\
p\leq x^{\frac{1}{3}}}}\frac{x/p}{(\log(x/p))^2}\ll\frac{x\log\log
x}{(\log x)^2}.
$$
This concludes the proof of (\ref{selbergp2}).

Let us return to the proof of (\ref{r2}). Clearly
$$
\sum_{n\leq x}r(n)^2\leq2\sum_{\substack{p_1,p_2\in\P\\
p_2\leq p_1\leq\log x/\log 2}}|\{q_1\in\P_2^*:\,
2^{p_1}-2^{p_2}+q_1\in\P_2^*\cap[1,x]\}|.
$$
If $p_1=p_2$, then
$$
\sum_{q_1\in\P_2^*\cap[1,x]}|\{q_2\in\P_2^*\cap[1,x]:\,
q_2=2^{p_1}-2^{p_2}+q_1\}|=\P_2^*(x)\ll\frac{x\log\log x}{\log x}.
$$
And if $p_1\not=p_2$, then
$$
\sum_{q_1\in\P_2^*\cap[1,x]}|\{q_2\in P_2^*\cap[1,x]:\,
q_2=2^{p_1}-2^{p_2}+q_1\}|\ll\frac{x(\log\log x)^2}{(\log
x)^2}\prod_{p\mid (2^{p_1-p_2}-1)}\bigg(1+\frac{1}{p}\bigg).
$$
Hence, by (3) we have
\begin{align*}
\sum_{n\leq x}r(n)^2\ll&\P(\frac{\log x}{\log 2})\frac{x\log\log
x}{\log x}+\frac{x(\log\log x)^2}{(\log
x)^2}\sum_{\substack{p_1,p_2\in\P\\
p_2<p_1\leq\frac{\log x}{\log 2}}}\prod_{p\mid
(2^{p_1-p_2}-1)}\bigg(1+\frac{1}{p}\bigg)\\
\ll&\frac{\log x}{\log\log x}\cdot\frac{x\log\log x}{\log
x}+\frac{x(\log\log x)^2}{(\log x)^2}\sum_{\substack{k\leq
\frac{\log x}{\log 2}}}2\prod_{p\mid
(2^{k}-1)}\bigg(1+\frac{1}{p}\bigg)\sum_{\substack{p_1,p_2\in\P\\
p_2<p_1\leq\frac{\log x}{\log 2}\\p_1-p_2=k}}1\\
\ll&x+\frac{x(\log\log x)^2}{(\log x)^2}\cdot\frac{2\log
x}{(\log\log x)^2}\sum_{\substack{0<k\leq \frac{\log x}{\log
2}}}\prod_{p\mid (2^{k}-1)}\bigg(1+\frac{1}{p}\bigg)\prod_{p\mid
k}\bigg(1+\frac{1}{p}\bigg).
\end{align*}
For any positive odd integer $d$, let $e(d)$ denote the least
positive integer such that $2^{e(d)}\equiv 1\pmod{d}$. Then
$2^k\equiv 1\pmod{d}$ if and only if $e(d)\mid k$. Now
\begin{align*}
\sum_{n\leq x}r(n)^2\ll&x+\frac{2x}{\log x}\sum_{\substack{0<k\leq
\frac{\log x}{\log 2}}}\prod_{p\mid
k}\bigg(1+\frac{1}{p}\bigg)\sum_{\substack{d\mid (2^{k}-1)\\
d\text{ square-free}}}\frac{1}{d}\\
=&x+\frac{2x}{\log x}\sum_{\substack{d\text{ square-free}\\ 2\nmid d}}\frac{1}{d}\sum_{\substack{0<k\leq \frac{\log x}{\log 2}\\
e(d)\mid k}}\prod_{p\mid
k}\bigg(1+\frac{1}{p}\bigg)\\
=&x+\frac{2x}{\log x}\sum_{\substack{d\text{ square-free}\\ 2\nmid d}}\frac{1}{d}\sum_{d'\text{ square-free}}\frac{1}{d'}\sum_{\substack{0<k\leq \frac{\log x}{\log 2}\\
e(d)\mid k, d'\mid k}}1\\
\leq&x+\frac{2x}{\log x}\cdot\frac{\log x}{\log
2}\sum_{\substack{d, d'\text{ square-free}\\ 2\nmid
d}}\frac{1}{dd'[e(d),d']}.
\end{align*}
Our final task is to show that the series
$$
\sum_{\substack{d, d'\text{ square-free}\\ 2\nmid
d}}\frac{1}{dd'[e(d),d']}
$$
converges. Clearly
$$
\sum_{\substack{d, d'\text{ square-free}\\ 2\nmid
d}}\frac{1}{dd'[e(d),d']}=\sum_{k>0}\sum_{d'\text{
square-free}}\frac{1}{d'[k,d']}\sum_{\substack{d\text{
square-free}\\ e(d)=k}}\frac{1}{d}.
$$
Let
$$
W(x)=\sum_{0<k\leq x}\sum_{\substack{d\text{ square-free}\\
e(d)=k}}\frac{1}{d}.
$$
With help of the arguments of Romanov (cf. \cite{Romanov34},
\cite[pp. 203]{Nathanson96}), we know that $W(x)\ll\log x$. And
\begin{align*}
\sum_{d'\text{
square-free}}\frac{1}{d'[k,d']}=&\frac{1}{k}\prod_{p\in\P, p\mid
k}\bigg(1+\frac{1}{p}\bigg)\prod_{p\in\P, p\nmid
k}\bigg(1+\frac{1}{p^2}\bigg)\\
\ll&\frac{1}{k}\prod_{p\in\P, p\mid
k}\bigg(1+\frac{1}{p}\bigg)\leq\frac{1}{\phi(k)}\ll
k^{-\frac{2}{3}}.
\end{align*}
Hence
\begin{align*}
\sum_{\substack{d, d'\text{ square-free}\\ 2\nmid
d}}\frac{1}{dd'[e(d),d']}\ll&\int_{\frac{1}{2}}^\infty
x^{-\frac{2}{3}}d
W(x)=\int_{\frac{1}{2}}^\infty \frac{2W(x)}{3x^{\frac{5}{3}}}dx+O(1)\\
\ll&\int_{\frac{1}{2}}^\infty \frac{\log
x}{x^{\frac{5}{3}}}dx+O(1)\ll 1.
\end{align*}
All are done.
\end{proof}

\begin{Rem}
Professor Y.-G. Chen told the second author two of his
conjectures:
\begin{Conj}
Let $A$ and $B$ be two subsets of positive integers. If there
exists a constant $c>0$ such that $A(\log x /\log2)B(x)>cx$ for
all sufficiently large $x$, then the set $\{2^a+b : a\in A, b\in
B\}$ has the positive lower asymptotic density.
\end{Conj}
\begin{Conj}
Let $A$ and $B$ be two subsets of positive integers. If there
exists a constant $c>0$ such that $A(\log x /\log2)B(x)>cx$ for
infinitely many positive integers $x$, then the set $\{2^a+b :
a\in A, b\in B\}$ has the positive upper asymptotic
density.
\end{Conj}
\end{Rem}
\begin{Ack} We thank Professor Yong-Gao Chen for his helpful
discussions on this paper.
\end{Ack}

\end{document}